\documentclass[a4paper,11pt]{amsart}
\usepackage[english]{babel}
\usepackage{amssymb}
\usepackage{amscd}
\usepackage{mathtools}
\usepackage[paper=a4paper,left=3cm,right=3cm,top=2cm,bottom=2.0cm]{geometry}
\usepackage{enumerate}
\usepackage{hyperref}
\usepackage{tikz}
\usepackage{tikz-cd}

\newtheorem{thm}{Theorem}[section]
\newtheorem*{thm-non}{Theorem}

\newtheorem{lem}[thm]{Lemma}
\newtheorem{prop}[thm]{Proposition}
\newtheorem{cor}[thm]{Corollary}
\theoremstyle{definition}
\newtheorem{defi}[thm]{Definition}
\theoremstyle{remark}
\newtheorem{rem}[thm]{Remark}

\DeclareMathOperator\Br{Br}
\DeclareMathOperator\End{End}
\DeclareMathOperator\Hom{Hom}
\DeclareMathOperator\Ext{Ext}
\DeclareMathOperator\Coh{Coh}
\DeclareMathOperator\id{id}
\DeclareMathOperator\M{M}
\DeclareMathOperator\Hilb{Hilb}
\DeclareMathOperator\Fix{Fix}
\DeclareMathOperator\tr{tr}
\DeclareMathOperator\Ima{Im}
\DeclareMathOperator\cok{Coker}
\DeclareMathOperator\Pic{Pic}

\DeclareMathOperator\modu{mod}
\DeclareMathOperator\supp{supp}
\DeclareMathOperator\Aut{Aut}

\begin{document}

\title{Rank one sheaves over quaternion algebras on Enriques surfaces}

\author{Fabian Reede}
\address{Institut f\"ur Algebraische Geometrie, Leibniz Universit\"at Hannover, Welfengarten 1, 30167 Hannover, Germany}
\email{reede@math.uni-hannover.de}

\subjclass[2010]{Primary: 14J60, Secondary: 14J28, 16H05}

\begin{abstract}
Let $X$ be an Enriques surface over the field of complex numbers. We prove that there exists a nontrivial quaternion algebra $\mathcal{A}$ on $X$. Then we study  the moduli scheme of torsion free $\mathcal{A}$-modules of rank one. Finally we prove that this moduli scheme is an \'{e}tale double cover of a Lagrangian subscheme in the corresponding moduli scheme on the associated covering K3 surface.
\end{abstract}

\maketitle

\section*{Introduction}
A \emph{noncommutative variety} is a pair $(X,\mathcal{A})$ consisting of a classical complex algebraic variety $X$ and a sheaf of noncommutative $\mathcal{O}_X$-algebras $\mathcal{A}$ of finite rank as an $\mathcal{O}_X$-module.

The algebras of interest in this article are \emph{Azumaya algebras}. These are algebras locally isomorphic to a matrix algebra $M_r(\mathcal{O}_X)$ with respect to the \'{e}tale topology. Especially interesting are the first nontrivial examples for $r=2$, the so called quaternion algebras, Azumaya algebras of rank four. These are generalizations of the classical quaternions $\mathbb{H}$.

Since the generic stalk of a quaternion algebra $\mathcal{A}$ is a central division ring over the function field of $X$, locally projective left $\mathcal{A}$-modules which are generically of rank one can be understood as line bundles on $(X,\mathcal{A})$. By \cite{hoff} there is a quasi-projective moduli scheme for these line bundles, a noncommutative Picard scheme, which can be compactified to a projective moduli scheme  $\M_{\mathcal{A}/X}$ by adding torsion free $\mathcal{A}$-modules generically of rank one.

We study in detail the situation of Enriques surfaces. We prove that every Enriques surface $X$ gives rise to a noncommutative Enriques surface $(X,\mathcal{A})$ with a quaternion algebra $\mathcal{A}$ on $X$. The main results of this article can be summarized as follows

\begin{thm-non}
Let $X$ be an Enriques surfaces, then there is a quaternion algebra $\mathcal{A}$ on $X$ representing the nontrivial class in $\Br(X)$. If $X$ is very general then
\begin{enumerate}[i)]
\item The moduli scheme $\M_{\mathcal{A}/X}$ of torsion free $\mathcal{A}$-modules of rank one is smooth.
\item Every torsion free $\mathcal{A}$-module of rank one can be deformed into a locally projective $\mathcal{A}$-module, i.e. the locus $\M_{\mathcal{A}/X}^{lp}$ of locally projective $\mathcal{A}$-modules is dense in $\M_{\mathcal{A}/X}$.
\end{enumerate}
Let $\overline{X}$ be the universal covering K3 surface of $X$ and denote the pullback of the quaternion algebra to $\overline{X}$ by $\overline{\mathcal{A}}$. For fixed Chern classes $c_1$ and $c_2$ we have
\begin{enumerate}[i)]
\setcounter{enumi}{2}
\item $\M_{\mathcal{A}/X,c_1,c_2}$ is an \'{e}tale double cover of a Lagrangian subscheme $\mathcal{L}\subset \M_{\overline{\mathcal{A}}/\overline{X},\overline{c_1},\overline{c_2}}$.
\end{enumerate} 
\end{thm-non}

The structure of this paper is as follows. We compare properties of modules over an Azumaya algebra on a smooth projective variety $W$ to those of the pullbacks to an \'{e}tale double cover $\overline{W}$ in section \ref{1}. In section \ref{2} we prove that a classical descent result for modules on the double cover is also true in the noncommutative setting. We look at the existence of Azumaya algebras on Enriques surfaces in section \ref{3}. In the final section \ref{4} we study moduli schemes of sheaves generically of rank one on a noncommutative Enriques surface. Many of the results in the last section are noncommutative analogues of results found by Kim in \cite{kim}. We work over the field of complex numbers $\mathbb{C}$.

\section{Modules over an Azumaya algebra and double coverings}\label{1}
In this section $W$ denotes a smooth projective complex variety of dimension $d$ together with a nontrivial 2-torsion line bundle $L$. By \cite[I.17]{hulek} there is an \'{e}tale Galois double cover
\begin{equation*}
q: \overline{W} \rightarrow W
\end{equation*}
with covering involution $\iota:\overline{W} \rightarrow\overline{W}$ such that
\begin{equation*}\label{push}
q_{*}\mathcal{O}_{\overline{W}}\cong\mathcal{O}_W\oplus L.
\end{equation*}

\begin{rem}
We make the following convention: for every coherent sheaf $E$ on $W$ we write $\overline{E}$ for the pullback to $\overline{W}$ along $q$, that is $\overline{E}:=q^{*}E$.
\end{rem}

\begin{defi}
A sheaf of $\mathcal{O}_W$-algebras $\mathcal{A}$ is called an \emph{Azumaya algebra} if it is locally free of finite rank and for every point $w\in W$ the fiber $\mathcal{A}(w)$ is a central simple algebra over the residue field $\mathbb{C}(w)$. Such a sheaf is called a quaternion algebra if $rk(\mathcal{A})=4$. Furthermore a coherent $\mathcal{O}_W$-module $E$ is said to be an Azumaya module or an $\mathcal{A}$-module if $E$ is also a left $\mathcal{A}$-module. 
\end{defi}

Azumaya algebras on $W$ are classified up to similarity by the Brauer group $\Br(W)$ of $W$. We say $\mathcal{A}$ is trivial if there is a locally free $\mathcal{O}_W$-module $P$ with $\mathcal{A}\cong \mathcal{E}nd_W(P)$ or equivalently $\left[\mathcal{A}\right]=0\in\Br(W)$. From now on, if not otherwise stated, by an Azumaya algebra $\mathcal{A}$ we mean a nontrivial Azumaya algebra. Furthermore we assume that there is a nontrivial Azumaya algebra $\mathcal{A}$ on $W$ such that $\overline{\mathcal{A}}$ is nontrivial on $\overline{W}$.  

\begin{lem}\label{hompull}
Assume $E$ and $F$ are $\mathcal{A}$-modules and $f : Z\rightarrow W$ is a flat morphism, then 
\begin{equation*}
\mathcal{H}om_{f^{*}\mathcal{A}}(f^{*}E,f^{*}F)\cong f^{*}{\mathcal{H}om_{\mathcal{A}}(E,F)}.
\end{equation*}
\end{lem}

\begin{proof}
First we note that by \cite[0.4.4.6]{gro} there is a natural morphism
\begin{equation*}
f^{*}{\mathcal{H}om_{\mathcal{A}}(E,F)}\rightarrow \mathcal{H}om_{f^{*}\mathcal{A}}(f^{*}E,f^{*}F).
\end{equation*}
So after a faithfully flat \'{e}tale base change we may assume that $\mathcal{A}$ is trivial. Then Morita equivalence for $\mathcal{A}=\mathcal{E}nd_W(P)$ reduces this problem to the case $\mathcal{A}=\mathcal{O}_W$. Now the lemma follows from \cite[0.6.7.6]{gro} since $f$ is flat by assumption.
\end{proof}
 
\begin{lem}\label{hom}
Assume $E$ and $F$ are $\mathcal{A}$-modules, then 
\begin{equation*}
\Hom_{\overline{\mathcal{A}}}(\overline{E},\overline{F})\cong \Hom_{\mathcal{A}}(E,F)\oplus \Hom_{\mathcal{A}}(E,F\otimes L).
\end{equation*}
\end{lem}

\begin{proof}
By the previous Lemma \ref{hompull} we have an isomorphism
\begin{equation*}
\mathcal{H}om_{\overline{\mathcal{A}}}(\overline{E},\overline{F})\cong \overline{\mathcal{H}om_{\mathcal{A}}(E,F)}.
\end{equation*}
This lemma is then a consequence of the following chain of isomorphisms, where the third line uses the projection formula for finite morphisms, \cite[Lemma 5.7]{ara}:
\begin{align*}
q_{*}\mathcal{H}om_{\overline{\mathcal{A}}}(\overline{E},\overline{F})&\cong q_{*}\overline{\mathcal{H}om_{\mathcal{A}}(E,F)}\\
&=q_{*}q^{*}\mathcal{H}om_{\mathcal{A}}(E,F)\\
&\cong \mathcal{H}om_{\mathcal{A}}(E,F)\otimes q_{*}\mathcal{O}_{\overline{W}}\\
&\cong \mathcal{H}om_{\mathcal{A}}(E,F)\oplus \mathcal{H}om_{\mathcal{A}}(E,F\otimes L).
\end{align*}
\end{proof}

\begin{cor}\label{homvan}
Assume $E$ is an $\mathcal{A}$-module. If $\overline{E}$ is a simple $\overline{\mathcal{A}}$-module, then $E$ is a simple $\mathcal{A}$-module and $\Hom_{\mathcal{A}}(E,E\otimes L)=0$.
\end{cor}
\begin{proof}
As $\overline{E}$ is a simple $\overline{\mathcal{A}}$-module, we have $\End_{\overline{\mathcal{A}}}(\overline{E})\cong\mathbb{C}$. Lemma \ref{hom} gives 
\begin{equation*}
\End_{\overline{\mathcal{A}}}(\overline{E})\cong\End_{\mathcal{A}}(E)\oplus\Hom_{\mathcal{A}}(E,E\otimes L)
\end{equation*}
and as $\id_E\in \End_{\mathcal{A}}(E)$ we find $\End_{\mathcal{A}}(E)\cong\mathbb{C}$ and $\Hom_{\mathcal{A}}(E,E\otimes L)=0$.
\end{proof}

\begin{prop}\cite[Proposition 3.5.]{hoff}\label{serre}
Assume $E$ and $F$ are $\mathcal{A}$-modules, then there is the following variant of Serre duality:
\begin{equation*}
\Ext^i_{\mathcal{A}}(E,F)\cong \left( \Ext^{d-i}_{\mathcal{A}}(F,E\otimes\omega_{W})\right) ^{\vee}.
\end{equation*}
\end{prop}

We assume now furthermore that $\dim W=2$. Denote the $\mathcal{O}_W$-double dual of $E$ by $E^{**}$.

\begin{lem}\label{double2}
Assume $E$ is an $\mathcal{A}$-module which is torsion free as an $\mathcal{O}_W$-module. If $\overline{E^{**}}$ is a simple $\overline{\mathcal{A}}$-module, then
\begin{equation*}
\Hom_{\mathcal{A}}(E,E^{**}\otimes L)=0.
\end{equation*}
\end{lem}
\begin{proof}
We first observe that there is an isomorphism
\begin{equation*}
 \End_{\mathcal{A}}(E^{**}) \cong \Hom_{\mathcal{A}}(E,E^{**}).
\end{equation*}
To see this, we note that there is an exact sequence of $\mathcal{A}$-modules
\begin{equation}\label{double}
\begin{tikzcd}
0 \arrow{r} & E \arrow{r} & E^{**} \arrow{r} & T \arrow{r} & 0 
\end{tikzcd}
\end{equation}
with $\dim\supp(T)=0$ as $E$ is torsion free and $\dim W=2$. It is known that $E^{**}$ is a locally free $\mathcal{O}_W$-module, hence a locally projective $\mathcal{A}$-module. This immediately implies $\Hom_{\mathcal{A}}(T,E^{**})=0$ since $T$ is torsion. Furthermore this also shows $\Ext^1_{\mathcal{A}}(T,E^{**})=0$ by using Proposition \ref{serre}, the local-to-global spectral sequence and the fact that $T$ is supported in dimension zero.
Applying $\Hom_{\mathcal{A}}(-,E^{**})$ to (\ref{double}) and using the vanishing results gives the desired isomorphism.

Using the same argument for $\overline{E}$ shows that we also have an isomorphism
\begin{equation*}
 \End_{\overline{\mathcal{A}}}(\overline{E^{**}}) \cong \Hom_{\overline{\mathcal{A}}}(\overline{E},\overline{E^{**}})
\end{equation*}
since $\overline{E^{**}}\cong\overline{E}^{**}$ by \cite[0.6.7.6.]{gro}.

We can now conclude as follows: by Lemma \ref{hom} we have
\begin{equation*}
\Hom_{\overline{\mathcal{A}}}(\overline{E},\overline{E^{**}})\cong \Hom_{\mathcal{A}}(E,E^{**})\oplus \Hom_{\mathcal{A}}(E,E^{**}\otimes L).
\end{equation*}
As $\overline{E^{**}}$ is simple then by the previous observation and Corollary \ref{homvan} we get
\begin{equation*}
\Hom_{\overline{\mathcal{A}}}(\overline{E},\overline{E^{**}})\cong\mathbb{C}\,\,\,\text{and}\,\,\, \Hom_{\mathcal{A}}(E,E^{**})\cong\mathbb{C}.
\end{equation*}
\end{proof}

\section{Noncommutative descent}\label{2}
We use the same notation as in the previous section. We have the \'{e}tale Galois double cover $q: \overline{W} \rightarrow W$ with $\text{Aut}(\overline{W}/W)$ generated by the covering involution $\iota$:
\begin{equation*}
\begin{tikzcd}
\overline{W}\arrow{rr}{\iota}\arrow{rd}[swap]{q} & & \overline{W}\arrow{ld}{q} \\
& W 
\end{tikzcd}  
\end{equation*}

\begin{defi}
We say a coherent sheaf $F$ of $\mathcal{O}_{\overline{W}}$-modules on $\overline{W}$ \emph{descends} to $W$, if there is a coherent sheaf $E$ of $\mathcal{O}_W$-modules on $W$ together with an isomorphism $F\cong\overline{E}$.
\end{defi}

Since $q:\overline{W}\rightarrow W$ is an \'{e}tale Galois double cover with $\Aut(\overline{W}/W)=\left\langle \iota \right\rangle \cong\mathbb{Z}/2\mathbb{Z}$ the descent condition for a coherent sheaf $F$ on $\overline{W}$, see   \cite[\href{https://stacks.math.columbia.edu/tag/0D1V}{Lemma 0D1V}]{stacks-project}, reduces to the existence of an isomorphism $\varphi_{\iota}: F\rightarrow \iota^{*}F$ such that (using $\varphi_{\iota^2}=\id$):
\begin{equation}\label{desccond}
\iota^{*}\varphi_{\iota}\circ\varphi_{\iota}=\id.
\end{equation}
But we have $\iota^{*}\varphi_{\iota}\circ\varphi_{\iota} : F\rightarrow \iota^{*}\iota^{*}F\cong F$. So, for example, if $F$ is simple, then any isomorphism $\varphi_{\iota}$ satisfies $\iota^{*}\varphi_{\iota}\circ\varphi_{\iota}\in \End_{\overline{W}}(F)=\mathbb{C}\cdot\id_F$. Hence after multiplication with an appropriate scalar, $\varphi_{\iota}$ satisfies (\ref{desccond}) and $F$ descends. Summing up:
\begin{prop}
Assume $F$ is a simple coherent $\mathcal{O}_{\overline{W}}$-module on $\overline{W}$ together with an isomorphism $F\cong\iota^{*}F$, then $F$ descends to $W$.
\end{prop}

In the rest of this section we want to prove a similar results for $\overline{\mathcal{A}}$-modules on $\overline{W}$. For this we need some notation: let $p: Y\rightarrow W$ be the Brauer-Severi variety of $\mathcal{A}$, see \cite{reede} for more information. By functoriality the Brauer-Severi variety $\overline{p}: \overline{Y}\rightarrow \overline{W}$ of $\overline{\mathcal{A}}$ is given by $\overline{Y}=Y\times_W\overline{W}$ and thus $\overline{q}: \overline{Y}\rightarrow Y$ is also an \'{e}tale Galois double cover with covering involution $\overline{\iota}$. All this fits in to the following diagram with both squares cartesian:
\begin{equation}\label{diag}
\begin{tikzcd}
\overline{Y} \arrow{r}{\overline{\iota}}\arrow{d}[swap]{\overline{p}} & \overline{Y} \arrow{r}{\overline{q}}\arrow{d}[swap]{\overline{p}} & Y \arrow{d}{p}\\
\overline{W} \arrow{r}{\iota} & \overline{W} \arrow{r}{q} & W
\end{tikzcd}
\end{equation}

The Brauer-Severi variety of $\mathcal{A}$ has the property that $\mathcal{A}_Y:=p^{*}\mathcal{A}$ is split, more exactly we have
\begin{equation*}
\mathcal{A}_Y^{op}\cong \mathcal{E}nd_Y(G)
\end{equation*}
for a locally free sheaf $G$ on $Y$, which can be described explicitly, see \cite[Remark 1.8]{reede}. 

In the following we will frequently use, without further mention, the fact that a coherent left $\mathcal{A}$-module is the same as a coherent right $\mathcal{A}^ {op}$-module. Denote these isomorphic categories by $\Coh_l(W,\mathcal{A})$ and $\Coh_r(W,\mathcal{A}^{op})$ respectively.

We also define
\begin{equation*}
\Coh(Y,W)=\left\lbrace E\in \Coh(Y)\,|\, p^{*}p_{*}(E\otimes G^{*})\xrightarrow{\,\cong\,} E\otimes G^{*} \right\rbrace.
\end{equation*}
Then by \cite[Lemma 1.10]{reede} we have the following equivalences
\begin{align*}
\phi: \Coh_r(W,\mathcal{A}^{op}) \rightarrow \Coh(Y,W),\,\,\, & E\mapsto p^{*}E\otimes_{\mathcal{A}_Y^{op}}G\\
\psi: \Coh(Y,W) \rightarrow \Coh_r(W,\mathcal{A}^{op}),\,\,\, & E\mapsto p_{*}(E\otimes G^{*}) 
\end{align*}
We have similar equivalences $\overline{\phi}$ and $\overline{\psi}$ involving $\overline{\mathcal{A}}_{\overline{Y}}^{op}\cong \mathcal{E}nd_{\overline{Y}}(\overline{q}^{*}G)$, $\overline{Y}$ and $\overline{W}$.

\begin{lem}\label{simple}
Assume $F$ is an $\overline{\mathcal{A}}$-module, then
\begin{equation*}
\End_{\overline{\mathcal{A}}}(F)\cong \End_{\overline{Y}}(\overline{\phi}(F)).
\end{equation*}
\end{lem}
\begin{proof}
Using $\mathcal{E}nd_{\overline{\mathcal{A}}}(F) = \mathcal{E}nd_{\overline{\mathcal{A}}^{op}}(F)$, the following chain of isomorphisms gives the result:
\begin{align*}
\mathcal{E}nd_{\overline{\mathcal{A}}^{op}}(F)&\cong \overline{p}_{*}\overline{p}^{*}\mathcal{E}nd_{\overline{\mathcal{A}}^{op}}(F) &&\text{by \cite[Lemma 1.6]{reede}}\\
&\cong\overline{p}_{*}\mathcal{E}nd_{\overline{\mathcal{A}}_{\overline{Y}}^{op}}(\overline{p}^{*}F)&&\text{by Lemma \ref{hompull}}\\
&\cong \overline{p}_{*}\mathcal{E}nd_{\mathcal{O}_{\overline{Y}}}(\overline{p}^{*}F\otimes_{\overline{\mathcal{A}}_{\overline{Y}}^{op}}\overline{q}^{*}G)&&\text{by Morita equivalence}\\
&=\overline{p}_{*}\mathcal{E}nd_{\mathcal{O}_{\overline{Y}}}(\overline{\phi}(F)).
\end{align*}
\end{proof}

\begin{lem}\label{xtoy}
Assume $F$ is a $\overline{\mathcal{A}}$-module such that there is an isomorphism $F\cong\iota^{*}F$ of $\overline{\mathcal{A}}$-modules, then $ \overline{\phi}(F)\cong\overline{\iota}^{*}(\overline{\phi}(F))$ as $\mathcal{O}_{\overline{Y}}$-modules.
\end{lem}
\begin{proof}
There are the following isomorphisms:
\begin{align*}
\overline{\iota}^{*}(\overline{\phi}(F))&=\overline{\iota}^{*}(\overline{p}^{*}F\otimes_{\overline{\mathcal{A}}_{\overline{Y}}^{op}}\overline{q}^{*}G)\\
&\cong \overline{\iota}^{*}\overline{p}^{*}F\otimes_{\overline{\iota}^{*}\overline{\mathcal{A}}_{\overline{Y}}^{op}}\overline{\iota}^{*}\overline{q}^{*}G&&\text{by \cite[4.3.3]{gro}}\\
&\cong \overline{p}^{*}\iota^{*}F\otimes_{\overline{\mathcal{A}}_{\overline{Y}}^{op}}\overline{q}^{*}G&&\text{by (\ref{diag})}\\
&\cong \overline{p}^{*}F\otimes_{\overline{\mathcal{A}}_{\overline{Y}}^{op}}\overline{q}^{*}G\\
&=\overline{\phi}(F).
\end{align*}
\end{proof}

\begin{lem}\label{pull}
Assume $F$ is a $\overline{\mathcal{A}}$-module such that there is $M\in Coh(Y)$ with $\overline{\phi}(F)\cong \overline{q}^{*}M$, then $M\in \Coh(Y,W)$.
\end{lem}
\begin{proof}
We have to prove that the canonical morphism 
\begin{equation}\label{canon}
p^{*}p_{*}(M\otimes G^{*})\rightarrow M\otimes G^{*}
\end{equation}
is an isomorphism. It is enough to prove this after the faithfully flat base change $\overline{q}: \overline{Y}\rightarrow Y$:
\begin{alignat*}{3}
 & \overline{q}^{*}(p^{*}p_{*}(M\otimes G^{*}))&&\rightarrow \overline{q}^{*}(M\otimes G^{*})& &\\
\cong\,\,\,\,&  \overline{p}^{*}q^{*}p_{*}(M\otimes G^{*})&&\rightarrow \overline{q}^{*}M\otimes (\overline{q}^{*}G)^{*}& &\,\,\,\,\,\text{by (\ref{diag}) and \cite[0.6.7.6]{gro}}\\
 \cong\,\,\,\,&\overline{p}^{*}\overline{p}_{*}(\overline{q}^{*}M\otimes \overline{q}^{*}G^{*}))&&\rightarrow \overline{q}^{*}M\otimes (\overline{q}^{*}G)^{*}& &\,\,\,\,\,\text{by (\ref{diag}) and  \cite[\href{https://stacks.math.columbia.edu/tag/02KH}{Lemma 02KH}]{stacks-project}}\\
\cong\,\,\,\,& \overline{p}^{*}\overline{p}_{*}(\overline{\phi}(F)\otimes \overline{q}^{*}G^{*}))&&\rightarrow \overline{\phi}(F)\otimes (\overline{q}^{*}G)^{*}& &
\end{alignat*}
But $\overline{\phi}(F)\in \Coh(\overline{Y},\overline{W})$, so the last morphism is an isomorphism, hence so is (\ref{canon}).
\end{proof}

We can now prove the main result of this section:

\begin{thm}\label{desc}
Assume $F$ is a simple $\overline{\mathcal{A}}$-module with an isomorphism $F\cong \iota^{*}F$ of $\overline{\mathcal{A}}$-modules, then there is an $\mathcal{A}$-module $E$ and an isomorphism of $\overline{\mathcal{A}}$-modules $F\cong\overline{E}$. 
\end{thm}
\begin{proof}
Since $F$ satisfies $F\cong\iota^{*}F$, by Lemma \ref{xtoy} we get an isomorphism $\overline{\phi}(F)\cong \overline{\iota}^{*}(\overline{\phi}(F))$. Since furthermore the $\mathcal{O}_{\overline{Y}}$-module $\overline{\phi}(F)$ is simple using Lemma \ref{simple}, it descends to $Y$, so $\overline{\phi}(F)\cong\overline{q}^{*}M$ for some coherent $\mathcal{O}_Y$-module $M$. But then $M\in \Coh(Y,W)$ due to Lemma \ref{pull}. Define $E:=\psi(M)$ then $E\in Coh_l(W,\mathcal{A})$ and $\overline{E}\cong F$ since:
\begin{align*}
\overline{E}=q^{*}\psi(M)&=q^{*}p_{*}(M\otimes G^{*})\cong \overline{p}_{*}(\overline{q}^{*}M\otimes (\overline{q}^{*}G)^{*})\cong \overline{p}_{*}(\overline{\phi}(F)\otimes (\overline{q}^{*}G)^{*})\cong F.
\end{align*}
\end{proof}

\section{Quaternion algebras on Enriques surfaces}\label{3}
\begin{defi}
A smooth projective surface $X$ is called an $\emph{Enriques surface}$ if it satisfies
\begin{itemize}
\item $H^1(X,\mathcal{O}_X)=0$
\item $\omega_X$ is 2-torsion, i.e. $\omega_X\neq \mathcal{O}_X$ but $\omega_X\otimes\omega_X\cong \mathcal{O}_X$.
\end{itemize}
\end{defi}
The 2-torsion element $\omega_X\in \Pic(X)$ induces an \'{e}tale Galois double cover 
\begin{equation*}
\pi: \overline{X}\rightarrow X.
\end{equation*}
It is well known that $\overline{X}$ is a K3 surface hence $\pi$ is a universal cover of $X$. Denote the associated involution by $\iota:\overline{X}\rightarrow\overline{X}$.

By results of Cossec and Dolgachev, see \cite[Theorem 1.1.3., Corollary 5.7.1.]{dolg} we have:
\begin{thm}
Assume $X$ is an Enriques surface over $\mathbb{C}$, then 
\begin{equation*}
\Br(X)\cong \mathbb{Z}/2\mathbb{Z}.
\end{equation*}
\end{thm}

This result shows that there is one nontrivial element $b_X$ in the Brauer group of an Enriques surface. The first question is if we can find a representative of this class in terms of Azumaya algebras.

\begin{prop}\label{repr}
The nontrivial element in the Brauer group of $X$ can be represented by a quaternion algebra $\mathcal{A}$ on $X$.
\end{prop}
\begin{proof}
The result of Cossec and Dolgachev shows that the nontrivial element $b_X\in\text{Br}(X)$ has order two.
As $X$ is smooth by \cite[Th\'{e}or\`{e}me 2.4.]{coll} the restriction to the generic point $\eta$ gives an injection
\begin{equation*}
r_\eta: \text{Br}(X)\hookrightarrow \text{Br}(\mathbb{C}(X)).
\end{equation*}
So the image $r_\eta(b_X)$ has order two in $\text{Br}(\mathbb{C}(X))$.

The field $\mathbb{C}(X)$ has property $C_2$, see \cite[II.4.5.(b)]{serre}. By a  result of Platonov (simultaneously found by Artin and Harris) the Brauer class $r_\eta(b_X)$ can be represented by a quaternion algebra $A$ over $\mathbb{C}(X)$, see \cite[Theorem 5.7]{sark} (\cite[Theorem 6.2.]{artin}).

Since the class $[A]$ comes from $\text{Br}(X)$ it is unramified at every point of codimension one in $X$, and thus by \cite[Th\'{e}or\`{e}me 2.5.]{coll} there is a quaternion algebra $\mathcal{A}$ on $X$ with $\mathcal{A}\otimes\mathbb{C}(X)=A$ such that $\left[ \mathcal{A}\right] =b_X$.
\end{proof}

One natural question to ask then: Is the pullback of the nontrivial class  still nontrivial in $\Br(\overline{X})$, i.e. is $\pi^{*}: \Br(X)\rightarrow \Br(\overline{X})$ injective? Beauville gives a complete answer to this question, see \cite[Corollary 4.3., Corollary 5.7., Corollary 6.5.]{beau}:

\begin{thm}
The morphism $\pi^{*}: \Br(X)\rightarrow \Br(\overline{X})$ is trivial if and only if there is $L\in\Pic(\overline{X})$ with $\iota^{*}L=L^{-1}$ and $c_1(L)^2\equiv 2\,\, (\modu 4)$. The surfaces $X$ with $\pi^{*}b_X=0$ form an infinite, countable union of (non-empty) hypersurfaces in the moduli space $\mathcal{M}$ of Enriques surfaces.
\end{thm}

Thus if $X$ is a very general Enriques surface (in the sense of the previous theorem) then the pullback of the quaternion algebra $\mathcal{A}$ constructed in Proposition \ref{repr} represents the nontrivial class $\pi^{*}b_X\in \Br(\overline{X})$.

\begin{rem}
For a description of the (non)triviality of $\pi^{*}: \Br(X)\rightarrow \Br(\overline{X})$ using lattice theory, group cohomology and the Hochschild-Serre spectral sequence, see \cite{mart}.  
\end{rem}

\section{Moduli schemes of sheaves over quaternion algebras}\label{4}
Assume $W$ is a smooth projective $d$-dimensional variety and $\mathcal{A}$ is an Azumaya algebra on $W$, then we can think of the pair $(W,\mathcal{A})$ as a noncommutative version of $W$. In this section, we want to study moduli schemes of sheaves on such noncommutative pairs.

\begin{defi}
A sheaf $E$ on $W$ is called a generically simple torsion free $\mathcal{A}$-module, if $E$ is a left $\mathcal{A}$-module such that $E$ is coherent and torsion free as a $\mathcal{O}_W$-module and the stalk $E_{\eta}$ over the generic point $\eta\in W$ is a simple module over $\mathcal{A}_{\eta}$. If furthermore $\mathcal{A}_{\eta}$ is a division ring over $\mathbb{C}(W)$ then such a module is also called a torsion free $\mathcal{A}$-module of rank one.
\end{defi}

\begin{rem}
A generically simple torsion free $\mathcal{A}$-module $E$ is simple, see \cite{hoff}.
\end{rem}

Apart from being simple, these modules share many properties with classical stable sheaves, for example we have

\begin{lem}\label{moduleiso}
Assume $E$ and $F$ are generically simple torsion free $\mathcal{A}$-modules with the same Chern classes, then $\Hom_{\mathcal{A}}(E,F)\neq 0$ implies $E\cong F$.
\end{lem}
\begin{proof}
A nontrivial $\mathcal{A}$-morphism $\phi$ must be generically bijective as $E$ and $F$ are generically simple. As $E$ and $F$ are torsion free this implies that $\phi$ is injective, so we get an exact sequence with $Q=\cok(\phi)$:
\begin{equation*}
\begin{tikzcd}
0 \arrow[r] & E \arrow[r,"\phi"] & F \arrow[r] & Q \arrow[r] & 0 
\end{tikzcd}
\end{equation*} 
But $E$ and $F$ have the same Chern classes, so $Q=0$ and hence $E\cong F$.
\end{proof}

By fixing the Hilbert polynomial $P$ of such sheaves, Hoffmann and Stuhler showed that these modules are classified by a moduli scheme, see \cite[Theorem 2.4. iii), iv)]{hoff}:

\begin{thm}
There is a projective moduli scheme $\M_{\mathcal{A}/W,P}$ classifying generically simple  torsion free $\mathcal{A}$-modules with Hilbert polynomial $P$ on $W$.
\end{thm}

We want to study these moduli schemes for a noncommutative Enriques surfaces $(X,\mathcal{A})$, where $X$ is a very general Enriques surface and $\mathcal{A}$ is a quaternion algebra representing the nontrivial class in $\Br(X)$. Note that the $\mathcal{O}_X$-rank of a torsion free $\mathcal{A}$-module of rank one is four in this case.  

We also have an associated noncommutative K3 surface $(\overline{X},\overline{\mathcal{A}})$. Now we first recall some facts about the moduli schemes for such pairs, see \cite[Theorem 3.6.]{hoff}:

\begin{thm}\label{hoff}
Let $\overline{X}$ be a K3 surface which is a double cover of a very general Enriques surface $X$ and let $\overline{\mathcal{A}}$ be the quaternion algebra coming from the quaternion algebra on $X$ which represents the nontrivial class in $\Br(X)$.
\begin{enumerate}[i)]
\item The moduli scheme $\M_{\overline{\mathcal{A}}/\overline{X}}$ of torsion free $\overline{\mathcal{A}}$-modules of rank one is smooth.
\item There is a nowhere degenerate alternating 2-form on the tangent bundle
of $\M_{\overline{\mathcal{A}}/\overline{X}}$
\item Every torsion free $\overline{\mathcal{A}}$-module of rank one can be deformed into a locally projective $\overline{\mathcal{A}}$-module, i.e. the locus $\M_{\overline{\mathcal{A}}/\overline{X}}^{lp}$ of locally projective $\mathcal{A}$-modules is dense in $\M_{\overline{\mathcal{A}}/\overline{X}}$.
\item For fixed Chern classes $\overline{c_1}$ and $\overline{c_2}$ we have 
\begin{equation*}
\dim \M_{\overline{\mathcal{A}}/\overline{X},\overline{c_1},\overline{c_2}}=\frac{\overline{\Delta}}{4}-c_2(\overline{\mathcal{A}})-6
\end{equation*}
where $\overline{\Delta}=8\overline{c_2}-3\overline{c_1}^2$ is the discriminant and $\overline{c_i}=\pi^{*}c_i$.
\end{enumerate}
\end{thm}

By using the $\overline{\mathcal{A}}$-Mukai vector we even get by \cite[Theorem 2.11]{reede}:

\begin{thm}
Let the pair $(\overline{X},\overline{\mathcal{A}})$ be as in Theorem \ref{hoff}. Assume $\overline{v}$ is a fixed primitive $\overline{\mathcal{A}}$-Mukai vector, then $\M_{\overline{\mathcal{A}}/\overline{X},\overline{v}}$ is an irreducible holomorphic symplectic manifold deformation equivalent to $\Hilb^{\frac{\overline{v}^2}{2}+1}(\overline{X})$.
\end{thm}

The covering involution $\iota: \overline{X}\rightarrow \overline{X}$ induces an involution  
\begin{equation*}
\iota^{*}: \M_{\overline{\mathcal{A}}/\overline{X},\overline{c_1},\overline{c_2}} \rightarrow \M_{\overline{\mathcal{A}}/\overline{X},\overline{c_1},\overline{c_2}},\,\,\, \left[F\right]\mapsto \left[\iota^{*}F\right]  
\end{equation*}

\begin{lem}\label{anti}
The involution $\iota^{*}$ is antisymplectic, that is if we denote the symplectic form on the tangent bundle of $\M_{\overline{\mathcal{A}}/\overline{X}}$  by $\omega$, then $\omega(\iota^{*}f_1,\iota^{*}f_2)=-\omega(f_1,f_2)$.  
\end{lem}

\begin{proof}
By \cite[Theorem 3.6. ii)]{hoff}, and similar to Mukai's construction, after the identification $T_{\left[F\right]}\M_{\overline{\mathcal{A}}/\overline{X}}\cong\Ext^1_{\overline{\mathcal{A}}}(F,F)$ the symplectic form is defined  by the Yoneda product
\begin{equation*}
\Ext^1_{\overline{\mathcal{A}}}(F,F) \times \Ext^1_{\overline{\mathcal{A}}}(F,F) \rightarrow \Ext^2_{\overline{\mathcal{A}}}(F,F).
\end{equation*}
composed with the trace map $\tr_{\overline{\mathcal{A}}}: \Ext^2_{\overline{\mathcal{A}}}(F,F) \rightarrow H^2(\overline{X},\mathcal{O}_{\overline{X}})$.

Using the functoriality of the Yoneda pairing (the cup product) we get the following commutative diagram
\begin{equation*}
\begin{tikzcd}
\Ext^1_{\overline{\mathcal{A}}}(F,F) \times \Ext^1_{\overline{\mathcal{A}}}(F,F) \arrow[r] \arrow[d,swap,shift right=3em,"\iota^{*}"]\arrow[d,shift left=3em,"\iota^{*}"]
& \Ext^2_{\overline{\mathcal{A}}}(F,F) \arrow[d,"\iota^{*}"] \\
\Ext^1_{\overline{\mathcal{A}}}(\iota^{*}F,\iota^{*}F) \times \Ext^1_{\overline{\mathcal{A}}}(\iota^{*}F,\iota^{*}F)
  \arrow[r]
& \Ext^2_{\overline{\mathcal{A}}}(\iota^{*}F,\iota^{*}F) 
\end{tikzcd}
\end{equation*}
According to the definition in \cite{hoff} the trace map $\tr_{\overline{\mathcal{A}}}$ is the composition of the forgetful functor from $\overline{\mathcal{A}}$-modules to $\mathcal{O}_{\overline{X}}$-modules and the usual trace map $\tr_{\mathcal{O}_{\overline{X}}}$, so $\tr_{\overline{\mathcal{A}}}$ is also functorial and we get the following commutative diagram
\begin{equation*}
\begin{tikzcd}
\Ext^2_{\overline{\mathcal{A}}}(F,F) \arrow{d}[swap]{\iota^{*}} \arrow{r}{\tr_{\overline{\mathcal{A}}}} & H^2(\overline{X},\mathcal{O}_{\overline{X}})\arrow{d}{\iota^{*}}\\
\Ext^2_{\overline{\mathcal{A}}}(\iota^{*}F,\iota^{*}F)\arrow{r}{\tr_{\overline{\mathcal{A}}}} & H^2(\overline{X},\mathcal{O}_{\overline{X}})
\end{tikzcd}
\end{equation*} 
But $\iota^{*} : H^2(\overline{X},\mathcal{O}_{\overline{X}}) \rightarrow H^2(\overline{X},\mathcal{O}_{\overline{X}})$ is multiplication by $-1$. This follows from the identification $H^2(\overline{X},\mathcal{O}_{\overline{X}})\cong\mathbb{C}$ by using $H^0(\overline{X},\omega_{\overline{X}})=\mathbb{C}\sigma$ with the symplectic form  $\sigma$ on $\overline{X}$ and the fact that $\iota^{*}$ is antisymplectic with respect to $\sigma$ as $H^0(X,\omega_X)=0$.

Putting both diagrams together, we see that $\iota^{*}$ is in fact antisymplectic. 
\end{proof}

\begin{cor}\label{isot}
The locus of fixed points of the involution 
\begin{equation*}
\Fix(\iota^{*})\subset \M_{\overline{\mathcal{A}}/\overline{X},\overline{c_1},\overline{c_2}}
\end{equation*}
is a smooth isotropic projective subscheme.
\end{cor}
\begin{proof}
$\Fix(\iota^{*})$ is smooth and projective by \cite[3.1., 3.4.]{edix}. The previous Lemma \ref{anti} shows that is also isotropic.
\end{proof}

For the rest of this section we need the following

\begin{rem}\label{eflat}
For a torsion free $\mathcal{A}$-module $E$ of rank one on $X$, the $\mathcal{A}$-modules $E^{**}$ and $E\otimes L$ for $L\in \Pic(X)$ are also torsion free of rank one. In addition $\overline{E}$ is a torsion free $\overline{\mathcal{A}}$-module of rank one on $\overline{X}$ since $\pi$ is flat.
\end{rem}

\begin{thm}\label{thm2}
Let $X$ be a very general Enriques surfaces and let $\mathcal{A}$ be a quaternion algebra on $X$ representing the nontrivial class in $\Br(X)$.
\begin{enumerate}[i)]
\item The moduli scheme $\M_{\mathcal{A}/X}$ of torsion free $\mathcal{A}$-modules of rank one is smooth.
\item Every torsion free $\mathcal{A}$-module of rank one can be deformed into a locally projective $\mathcal{A}$-module, i.e. the locus $\M_{\mathcal{A}/X}^{lp}$ of locally projective $\mathcal{A}$-modules is dense in $\M_{\mathcal{A}/X}$.
\item For fixed Chern classes $c_1$ and $c_2$ we have 
\begin{equation*}
\dim \M_{\mathcal{A}/X,c_1,c_2}=\frac{\Delta}{4}-c_2(\mathcal{A})-3
\end{equation*}
where $\Delta=8c_2-3c_1^2$ is the discriminant.
\end{enumerate}
\end{thm}
\begin{proof}
\begin{enumerate}[i)]
\item For a given point $[E]\in M_{\mathcal{A}/X}$ we have to show that all obstruction classes in $\Ext^2_{\mathcal{A}}(E,E)$ vanish. But by Proposition \ref{serre} we have:
\begin{equation*}
\Ext^2_{\mathcal{A}}(E,E)\cong \left( \Hom_{\mathcal{A}}(E,E\otimes\omega_X)\right)^{\vee}.
\end{equation*}
As $\overline{E}$ is a simple $\overline{\mathcal{A}}$-module, we get $\Hom_{\mathcal{A}}(E,E\otimes\omega_X)=0$ by Corollary \ref{homvan}. Thus all obstructions vanish and $M_{\mathcal{A}/X}$ is smooth at $\left[E\right]$.
\item The proof of \cite[Theorem 3.6.iii)]{hoff} carries over to our situtaion with one small change: the surjectivity of the connecting homomorphisms $\delta$ in the diagram:
\begin{equation*}
\begin{tikzcd}
\Ext^1_{\mathcal{A}}(E,E) \arrow{r}{\delta} & \Ext^2_{\mathcal{A}}(T,E) \arrow{r}{\pi^{*}}\arrow{d}{\iota_{*}} & \Ext^2_{\mathcal{A}}(E^{**},E)\\
 & \Ext^2_{\mathcal{A}}(T,E^{**}) \arrow[equal]{r} & \bigoplus\limits_{i=1}^l\Ext^2_{\mathcal{A}}(T_{x_i},E^{**})
\end{tikzcd}
\end{equation*} 
follows from the fact that
\begin{equation*}
\Ext^2_{\mathcal{A}}(E^{**},E)=0.
\end{equation*}
This vanishing can be seen as follows: using Proposition \ref{serre} we have
\begin{equation*}
\Ext^2_{\mathcal{A}}(E^{**},E)\cong\left(\Hom_{\mathcal{A}}(E,E^{**}\otimes\omega_X) \right)^{\vee}.
\end{equation*}
But the last space is zero by Lemma \ref{double2}. The rest of the proof works unaltered.
\item Using ii) is suffices to compute the dimension of
\begin{equation*}
T_{\left[E\right]}M_{\mathcal{A}/X} \cong \Ext^1_{\mathcal{A}}(E,E)\cong H^1(X,\mathcal{E}nd_{\mathcal{A}}(E))
\end{equation*}
for a locally projective $\mathcal{A}$-module $E$ of rank one. 

Again as in \cite[Theorem 3.6.iv)]{hoff} we have:
\begin{equation*}
c_1(\mathcal{E}nd_{\mathcal{A}}(E))=0 \,\,\,\,\text{and}\,\,\,c_2(\mathcal{E}nd_{\mathcal{A}}(E)))=\frac{\Delta}{4}-c_2(\mathcal{A})
\end{equation*}
where $\Delta$ is the discriminant of $E$. So by Hirzebruch-Riemann-Roch:
\begin{equation*}
\chi(X,\mathcal{E}nd_{\mathcal{A}}(E))=-\frac{\Delta}{4}+c_2(\mathcal{A})+4\chi(X,\mathcal{O}_X)
\end{equation*}
Using $\End_{\mathcal{A}}(E)\cong\mathbb{C}$, $\Ext^2_{\mathcal{A}}(E,E)=0$ and $\chi(X,\mathcal{O}_X)=1$ we get our result.
\end{enumerate}
\end{proof}

\begin{rem}\label{notiso}
The proof of i) also implies $E\ncong E\otimes\omega_X$ for all torsion free $\mathcal{A}$-modules of rank one.
\end{rem}

Similar to the involution $\iota$, using Remark \ref{eflat}, the projection $\pi:\overline{X}\rightarrow X$ induces a morphism
\begin{equation*}
\pi^{*}: \M_{\mathcal{A}/X,c_1,c_2} \rightarrow \M_{\overline{\mathcal{A}}/\overline{X},\overline{c_1},\overline{c_2}},\,\,\,\, \left[E\right] \mapsto \left[\overline{E}\right].  
\end{equation*}
Our goal is to understand this morphism:

\begin{thm}\label{moduli}
Let the pair $(X,\mathcal{A})$ be as in Theorem \ref{thm2}. The pullback map 
\begin{equation*}
\pi^{*}: \M_{\mathcal{A}/X,c_1,c_2} \rightarrow \M_{\overline{\mathcal{A}}/\overline{X},\overline{c_1},\overline{c_2}}
\end{equation*}
factors through $\Fix(\iota^{*})$ restricting to an \'{e}tale double cover
\begin{equation*}
\varphi: \M_{\mathcal{A}/X,c_1,c_2} \rightarrow \Fix(\iota^{*}).
\end{equation*}
\end{thm}

\begin{proof}
We have
\begin{equation*}
\iota^{*}\overline{E}=\iota^{*}\pi^{*}E\cong(\pi\circ\iota)^{*}E=\pi^{*}E=\overline{E}.
\end{equation*}
So $\Ima(\pi^{*})\subset \Fix(\iota^{*})$ and hence $\pi^{*}$ factors through $\Fix(\iota^{*})$ giving rise to 
\begin{equation*}
\varphi: \M_{\mathcal{A}/X,c_1,c_2} \rightarrow \Fix(\iota^{*}).
\end{equation*}
By Theorem \ref{desc} we also have $\Fix(\iota^{*})\subset \Ima(\pi^{*})$. So $\Ima(\pi^{*})=\Fix(\iota^{*})$ and the morphism $\varphi$ is surjective.

Assume $\varphi(\left[E\right] )=\varphi(\left[F\right] )$ that is $\overline{E}\cong\overline{F}$ and $\Hom_{\overline{\mathcal{A}}}(\overline{E},\overline{F})\neq 0$. Then Lemma \ref{hom} says
\begin{equation*}
\Hom_{\overline{\mathcal{A}}}(\overline{E},\overline{F}) \cong \Hom_{\mathcal{A}}(E,F)\oplus \Hom_{\mathcal{A}}(E,F\otimes\omega_X) 
\end{equation*}
and so by Lemma \ref{moduleiso} and Remark \ref{eflat} we have 
\begin{equation*}
E\cong F\,\,\,\,\text{or}\,\,\,\, E\cong F\otimes\omega_X
\end{equation*}
but not both by Remark \ref{notiso}. So $\varphi$ is an unramified $2:1$-morphism. Moreover the computations also shows that $\varphi$ is a flat morphism by \cite[Lemma, p.675]{schaps}, hence $\varphi$ is \'{e}tale.
\end{proof}

\begin{cor}
The locus of fixed points of the involution 
\begin{equation*}
\Fix(\iota^{*})\subset \M_{\overline{\mathcal{A}}/\overline{X},\overline{c_1},\overline{c_2}}
\end{equation*}
is a Lagrangian subscheme.
\end{cor}
\begin{proof}
The previous theorem \ref{moduli} shows
\begin{equation*}
\dim \Fix(\iota^{*})=\dim \M_{\mathcal{A}/X,c_1,c_2}.
\end{equation*}
On the other hand by Theorem \ref{hoff} and the fact that $\pi$ is of degree 2 we have
\begin{equation*}
\dim \M_{\overline{\mathcal{A}}/\overline{X},\overline{c_1},\overline{c_2}}=\frac{\overline{\Delta}}{4}-c_2(\overline{\mathcal{A}})-6=2(\frac{\Delta}{4}-c_2(\mathcal{A})-3)=2\dim \M_{\mathcal{A}/X,c_1,c_2}
\end{equation*}
Both results together give
\begin{equation*}
\dim \Fix(\iota^{*})=\frac{1}{2}\dim \M_{\overline{\mathcal{A}}/\overline{X},\overline{c_1},\overline{c_2}}
\end{equation*}
By Corollary \ref{isot} we already know that $\Fix(\iota^{*})$ is an isotropic subscheme, so it is in fact a Lagrangian subscheme of $\M_{\overline{\mathcal{A}}/\overline{X},\overline{c_1},\overline{c_2}}$.
\end{proof}


\end{document}